\documentclass[10pt,a4paper]{article}
\usepackage[latin1]{inputenc}
\usepackage{amsmath, amsfonts, amssymb, amsthm, latexsym}
\usepackage[left = 1 in, right = 1 in]{geometry}
\usepackage{fancyhdr}

\newcommand{\R}{\mathbb{R}}
\newcommand{\Z}{\mathbb{Z}}
\newcommand{\Nz}{\mathbb{N}_0}

\newcommand{\phelp}{p}
\newcommand{\pinv}{Q}
\newcommand{\pinvhelp}{q}
\newcommand{\pinvequal}{\pi}
\newcommand{\littleO}{o}
\newcommand{\bigO}[1]{\mathcal{O} \left( #1 \right)}

\makeatletter
\newtheorem*{rep@theorem}{\rep@title}
\newcommand{\newreptheorem}[2]{%
\newenvironment{rep#1}[1]{%
 \def\rep@title{#2 \ref{##1}}%
 \begin{rep@theorem}}%
 {\end{rep@theorem}}}
\makeatother

\newreptheorem{theorem}{Theorem}
\newreptheorem{lemma}{Lemma}

\newtheorem{theorem}{Theorem}
\newtheorem{lemma}[theorem]{Lemma}
\newtheorem{proposition}[theorem]{Proposition}
\newtheorem{corollary}[theorem]{Corollary}

\newcommand{\ignore}[1]{}
\newcommand{\secref}[1]{\textbf{Section #1}}

\title{Sets with High Volume and Low Perimeter}
\author{Patrick Devlin -- PatrickDevlin21@gmail.com}

\begin{document}
\maketitle
\ignore{\begin{abstract}
In this paper, we explore a certain variation of the so-called ``isoperimetric problem" in which integer subsets take the role of geometric figures.  BLAH
\end{abstract}}
\section{Introduction}\label{section introduction}
One of the most widely-known classical geometry problems is the so-called \textit{isoperimetric problem}, one equivalent variation of which is:
\begin{quote}
If a figure in the plane has area $A$, what is the smallest possible value for its perimeter?
\end{quote}
In the Euclidean plane, the optimal configuration is a circle, implying that any figure with area $A$ has perimeter at least $2 \sqrt{A\pi}$, and this lower bound may be obtained if and only if the figure is a circle.
\paragraph*{}In 2011, Miller et al.~\cite{Miller} extended the isoperimetric problem in a new direction, in which integer subsets took the role of geometric figures.  For any integer subset $A$, they defined its \textit{volume} as the sum over all its elements, and they defined its \textit{perimeter} as the sum of all elements $x \in A$ such that $\{x-1, x+1\} \not \subset A$.  Thus, the volume can be thought of as the sum of all the elements of $A$, and the perimeter can be thought of as the sum of all the elements on the ``boundary" of $A$ (that is to say, the elements of $A$ whose successor and predecessor are not both in $A$).
\paragraph*{}The main focus of \cite{Miller} was to examine the relationships between a set's perimeter and its volume.  More specifically, the authors wanted to answer the corresponding ``isoperimetric question"\footnote{They focused on this question in particular because it turns out that all of the related extremal questions are trivial.}:
\begin{quote}
If a subset of $\{0, 1, \ldots \}$ has volume $n$, what is the smallest possible value for its perimeter?
\end{quote}
Adopting their notation, we will let $P(n)$ denote this value through the duration of this paper.
\paragraph*{}Because their work is so recent, Miller et al. are apparently the only ones who have published on this variation of the isoperimetric problem or on the function $P(n)$.  Their work was to provide bounds for $P(n)$, by which they were able to determine its asymptotic behavior.  Specifically, their main result was
\begin{theorem}\label{theirResult}
\emph{(Miller et al., 2011)} Let $P(n)$ be as defined.  Then $P(n) \thicksim \sqrt{2}n^{1/2}$.  Moreover, for all $n \geq 1$,
\begin{equation}\label{theirBounds}
\sqrt{2}n^{1/2} - 1/2 < P(n) < \sqrt{2} n^{1/2} + (2 n^{1/4} + 8) \log _2 \log _2 n + 58.
\end{equation}
\end{theorem}
\paragraph*{}Their proof of the lower bound will be reproduced in following sections.  However, their proof of the upper bound was found by a construction argument, which we will not reproduce here since we will analytically derive a tighter bound in \textbf{Theorem \ref{myBounds}}.
\paragraph*{}Beyond the inequalities in \eqref{theirBounds} provided by Miller et al., nothing else has been published on $P(n)$ except for some values for small $n$.  It should be noted that \cite{Miller} provides very good bounds on a related function, in which the sets of interest are allowed to have both negative as well as positive elements.  However, this result was also obtained by a construction argument, and it is not relevant to this paper.
\subsection*{Outline of Results}
In this paper, we focus on improving the few results known on $P(n)$, including deriving multiple exact formulas and developing an understanding of its interesting long-term behavior.  Many of these results are stated in terms of an intimately related function, $Q(n)$, which may be briefly defined as\footnote{More adequate introductions of this function are given in subsequent sections.}
\[
Q(n) := \min_{A \subseteq \{0, 1, \ldots\}} \Big \{ per(A^{c}) : vol(A) = n \Big \}.
\]
Since it proves to be so closely related to $P(n)$, we also provide results on $Q(n)$ throughout the paper.
\paragraph*{}We begin in \secref{\ref{section preliminary}} by proving several prelimary lemmas including those used in \cite{Miller}.  Then in \secref{\ref{section first recurrences}}, we define auxilary functions, with which we combinatorially derive several recursive formulas for $P(n)$.  We then introduce the function $Q(n)$ and derive similar recursive formulas for it as well.
\paragraph*{}In \secref{\ref{section second recurrences}}, we relate the functions $P(n)$ and $Q(n)$ by providing yet more recurrence relations for both of them, from which we see that each function completely determines the other.  With this in place, we move on to \secref{\ref{section analysis of recurrences}}, in which we use these recurrences to determine several analytic results for $P(n)$ and $Q(n)$, including upper and lower bounds and derivations of their asymptotic behavior.
\paragraph*{}Our work then culminates in \secref{\ref{section good recurrences}}, in which we state and prove the main results of the paper.  By appealing to our analytic bounds on $P(n)$ and $Q(n)$, we show that the recurrences of \secref{\ref{section second recurrences}} admit certain drastic simplifications.  With this, we derive several elegant reccurence relations and quasi-explicit representations for $P(n)$ and $Q(n)$, and we also briefly remark on the intricate fractal-like symmetry of these functions.
\paragraph*{}More specifically, our main result is that for all sufficiently large values of $n$, the functions $P(n)$ and $Q(n)$ admit the following useful and satisfying recurrence relations, where $f(n)$ and $g(n)$ are simple explicitly defined functions related to the distribution of the triangular numbers:
\begin{reptheorem}{bestResult}
Let $P(n)$ and $Q(n)$ be as given.  Then there exists some integer $N$ such that for all $n\geq N$, 
\begin{eqnarray*}
P(n) &=& f(n) + Q(g(n)) \qquad \qquad \text{and}\\
Q(n) &=& 1 + f(n) + P(g(n)),
\end{eqnarray*}
where $f(n)$ and $g(n)$ are as defined.
\end{reptheorem}
\paragraph*{}We conclude in \secref{\ref{section conclusion}} by noting applications in the design of algorithms related to this problem and with some open questions for future research.

\section{Definitions and Notation}\label{section definitions}
For the reader's possible convenience, a list of formal definitions used throughout the paper is given here.  Since each notion is adequately defined as it naturally arises in the paper, we suggest that the reader not burden himself with the details of these definitions at this time and instead use them only for possible reference as he procedes.
\begin{itemize}
\item Let $A$ be an integer subset.  Then we define the \textit{boundary} of $A$ as
\[
\partial A = \{z \in A : \{ z-1, z+1\} \not \subseteq A\}.
\]
In words, this is the set of elements of $A$ whose successor or predecessor is not in $A$.
\item Let $A$ be an integer subset.  Then the \textit{volume} and \textit{perimeter} of $A$ are defined as
\[
vol(A) = \sum_{z \in A} z, \qquad \text{and} \qquad per(A) = \sum_{z \in \partial A} z,
\]
respectively.  These values may be infinite (although they will not be in our considerations).  The volume and perimeter of the empty set is 0.
\item The function $P(n)$ can be formally defined as
\[
P(n) = \min_{A \subseteq \{0, 1, \ldots\}} \Big \{ per(A) : vol(A) = n \Big \}.
\]
\item Let $A \subseteq \{0, 1, \ldots \}$.  Then we define the \textit{complement} of $A$ as
\[
A^{c} = \{0, 1, \ldots \} \setminus A = \{z \in \{0, 1, \ldots \} : z \notin A\}.
\]
\item The function $Q(n)$ can be formally defined as
\[
Q(n) = \min_{A \subseteq \{0, 1, \ldots\}} \Big \{ per(A^{c}) : vol(A) = n \Big \}.
\]
\item The helper functions $p(n;k)$ and $q(n;k)$ can be formally defined as
\[
p(n;k) = \min_{A \subseteq \{0, 1, \ldots , k\}} \Big \{per(A) : vol(A) = n \Big \}, \quad \qquad q(n;k) = \min_{A \subseteq \{0, 1, \ldots , k\}} \Big \{per(A^{c}) : vol(A) = n \Big \}.
\]
\item The special helper function $\pinvequal (n;k)$ can be formally defined as
\[
\pinvequal (n;k) = \min_{A \subseteq \{0, 1, \ldots, k\}} \Big \{per(A^c) : vol(A) = n, \quad \text{and} \quad k \in A \Big \}.
\]
\item For all $n \geq 0$, we define the functions $f(n)$ and $g(n)$ as
\[
f(n) = \left \lceil (\sqrt{8n + 1} - 1)/2 \right \rceil, \qquad \text{and} \qquad g(n) = \dfrac{f(n) [f(n) + 1]}{2} - n,
\]
where $\lceil x \rceil$ denotes the ceiling function.  In \textbf{Proposition \ref{fReps}}, we provide several equivalent representations for the function $f(n)$.
\item For all $n \geq 0$ and for all $N$, we define $\phi (n; N) = \phi (n)$ as the smallest nonnegative integer $i$ such that
\[
g^{i}(n) \leq N,
\]
where $g^{i}(n)$ denotes the $i$-fold composition of $g$ evaluated at $n$.
\end{itemize}

\section{Preliminary Results}\label{section preliminary}
The following lemma is used throughout \cite{Miller} and is essential in proving their lower bound on $P(n)$.
\begin{lemma}\label{basicInequality}
\emph{(Miller et al., 2011)} Assume $A$ is a finite nonempty subset of $\{0, 1, \ldots \}$, and let $m$ denote its maximum element.  Then
\[
m \leq per(A) \leq vol(A) \leq \dfrac{m(m+1)}{2}.
\]
\end{lemma}
\begin{proof}
The fact that $per(A) \leq vol(A)$ is clear from the fact that all the elements of $A$ are nonnegative and $\partial A \subseteq A$.
\paragraph*{}Now because $m$ is the maximum element of $A$, we know that $m+1 \notin A$.  Therefore $m \in \partial A$, which means that it is a term in the summation for $per(A)$.  Since all the other terms must be nonnegative, we have $m \leq per(A)$.
\paragraph*{}Finally, since $m$ is the maximum element of $A$, we know $A \subseteq \{0, 1, \ldots , m\}$, implying that
\[
vol(A) \leq 0 + 1 + 2 + \cdots + m = \dfrac{m(m+1)}{2},
\]
which completes the proof.
\end{proof}
Using the previous lemma, the following lower bound is immediately attained.
\begin{proposition}\label{lowerBoundForP}
Assume $A \subseteq \{0, 1, \ldots \}$ is $A$ is finite.  Then we have
\[
\sqrt{2vol(A)} -1/2 \leq \dfrac{-1 + \sqrt{1+8 vol(A)}}{2} \leq per(A).
\]
Moreover, for any positive integer $n$, this implies
\[
\sqrt{2} n^{1/2} - 1/2 \leq P(n).
\]
\end{proposition}
As stated before, except for the previously mentioned construction yielding an upper bound on $P(n)$, the previous two results seem to be all that is currently known about $P(n)$.  The remainder of the paper is devoted to original results.

\subsection*{Miscellaneous Lemmas}
\begin{lemma}\label{complementInequality}
Assume $A$ is a subset of $\{0, 1, \ldots \}$ such that $per(A)$ is finite.  Then $per(A^c)$ is also finite, and
\[
\dfrac{per(A) - 1}{2} \leq per(A^c) \leq 2 \cdot per(A) + 1.
\]
Moreover, for all $n \geq 1$, we have
\[
\dfrac{P(n)-1}{2} \leq Q(n) \leq 2 \cdot P(n) + 1.
\]
\end{lemma}
\begin{proof}
We will prove that $per(A ^c) \leq 2 \cdot per(A) + 1$, and by switching the roles of $A$ and $A^c$, this will complete the proof of the first inequality.  The second inequality then follows directly from the definitions.
\paragraph*{}Let $A \subseteq \Nz$ be such that $per(A)$ is finite.  Then it follows that $\partial A$ is a finite set.  Now we know that an element $z \in A^c$ is in $\partial A^c$ if and only if $z+1$ or $z-1$ is an element of $\partial A$.  Therefore, we have
\[
\partial A^c \subseteq \{z-1, z+1 : z \in \partial A \} \setminus \{-1\},
\]
where the element $-1$ should not be counted in the event that $0 \in \partial A$.  Therefore, we have that
\[
per(A^c) = \sum_{z \in \partial A^c} z \leq \left( \sum_{z \in \partial A} (z-1) + (z+1) \right ) - (-1) = 2 \sum_{z \in \partial A} z + 1 = 2 per(A) + 1,
\]
which completes the proof.
\end{proof}
This establishes that although $P(n)$ and $Q(n)$ are different, they are still in some sense ``close".
\begin{lemma}
Assume $A$ is a finite nonempty subset of $\{0, 1, \ldots \}$, and let $m$ denote its maximum element.  Then
\[
m+1 \leq per(A^c)
\]
with equality if and only if $\{1, \ldots, m\} \subseteq A$.
\end{lemma}
\begin{proof}
Let $A$ be as given.  Then $m \in A$, but we know $m+1 \notin A$.  Therefore, $m+1 \in \partial A^c$ implying that $m+1 \leq per(A^c)$.  Now since $m+1 \in \partial A^c$, we know that $m+1 = per(A^c)$ if and only if $\partial A^c$ is equal to either $\{m+1\}$ or $\{0, m+1\}$.  But this happens if and only if $\{1, 2, \ldots , m\} \subseteq A$, as desired.
\end{proof}
\begin{proposition}\label{lowerBoundForQ}
Assume $A \subseteq \{0, 1, \ldots \}$ is $A$ is finite.  Then we have
\[
\sqrt{2vol(A)} + 1/2 \leq \dfrac{-1 + \sqrt{1+8 vol(A)}}{2} + 1 \leq per(A^c).
\]
Moreover, for any positive integer $n$, this implies
\[
\sqrt{2} n^{1/2} + 1/2 \leq Q(n).
\]
\end{proposition}
\begin{proof}
This follows from the previous lemma in the same was as \textbf{Proposition \ref{lowerBoundForP}}.
\end{proof}

\section{Recurrence Relations using Auxilary Functions} \label{section first recurrences}
The first set of recurrence relations the author found is detailed below.  Although the relations that are described in the section after this are much simpler, more useful, and elegant, the relations presented here follow naturally, and they motivate the introduction of the auxilary functions needed in the remainder of the paper.  Moreover, because of their convenient structure, these relations are used extensively in the design of algorithms for computing values; this application is briefly discussed in the end of the paper.
\subsection*{First Recurrence for $P(n)$}
As is often the case in analyzing discrete functions, we may obtain an exact recurrence relation for $P(n)$ in terms of some related auxillary function.  We will use this strategy several times in the remainder of the paper.  In our case, recall that $P(n)$ is the minimum perimeter among all subsets of $\{0, 1, \ldots \}$ having volume $n$.  Now observe the following decomposition for $P(n)$
\begin{eqnarray*}
P(n) &=& \min_{A \subseteq \{0, 1, \ldots \}} \Big \{ per(A) : vol(A) = n \Big \}\\
&=& \min \Bigg \{\min_{A \subseteq \{0\}} \Big \{ per(A) : vol(A) = n \Big \}, \min_{A \subseteq \{0, 1\}} \Big \{ per(A) : vol(A) = n \Big \}, \\
& & \qquad \qquad \min_{A \subseteq \{0, 1, 2\}} \Big \{ per(A) : vol(A) = n \Big \},  \min_{A \subseteq \{0, 1, 2, 3\}} \Big \{ per(A) : vol(A) = n \Big \}, \ldots \Bigg \}\\
&=& \min_{k \in \{0, 1, \ldots \}} \Bigg \{ \min_{A \subseteq \{0, 1, \ldots , k\}} \Big \{ per(A) : vol(A) = n \Big \} \Bigg \}.
\end{eqnarray*}
Based on this representation, we define the auxilary function, $\phelp (n;k)$, given by
\[
\phelp (n; k) = \min_{A \subseteq \{0, 1, \ldots , k\}} \Big \{ per(A) : vol(A) = n \Big \}.
\]
With this notation, the previous equation can be written as
\[
P(n) = \min_{k \in \{0, 1, \ldots \}} \Bigg \{ \min_{A \subseteq \{0, 1, \ldots , k\}} \Big \{ per(A) : vol(A) = n \Big \} \Bigg \} = \min_{k \in \{0, 1, \ldots \}} \Bigg \{ \phelp (n;k) \Bigg \}.
\]
\paragraph*{}Now from its definition, it is clear that for all fixed $n$, the function $\phelp (n; k)$ is monotonically decreasing with $k$.  Moreover, for all $K \geq n$, we have $\phelp (n; K) = \phelp (n; n)$ since any subset of $\{0, 1, \ldots \}$ having volume $n$ must necessarily be a subset of $\{0, 1, 2, \ldots , n\}$.  Therefore, with these two facts, the above equation simplifies to
\begin{equation}\label{PRecHelp}
P(n) = \min_{k \in \{0, 1, \ldots \}} \Bigg \{ \phelp (n;k) \Bigg \} = \lim _{k \to \infty} \phelp (n; k) = \phelp (n; n).
\end{equation}
Thus, we now seek to obtain a recurrence for $\phelp (n; k)$, which will provide us with $P(n)$ by calculating $\phelp (n; n)$.
\paragraph*{}Now let $S(n; k)$ denote the set of all subsets of $\{0, 1, \ldots , k \}$ having volume $n$.  Then to obtain our desired recurrence for $\phelp (n;k)$, we will consider the following paritition of the set $S(n;k)$
\begin{eqnarray*}
S(n;k) &=& \{A \in S(n;k) : k \notin A \} \cup \{A \in S(n;k) : k \in A \quad \text{and} \quad k-1 \notin A\}\\
& & \qquad \cup \{A \in S(n;k) : \{k-1, k\} \subseteq A \quad \text{and} \quad k-2 \notin A\} \cup \cdots\\
& & \qquad \cup \{A \in S(n;k) : \{1, 2, \ldots , k\} \subseteq A \quad \text{and} \quad 0 \notin A\}\\
& & \qquad \cup \{A \in S(n;k) : \{0, 1, \ldots, k\} \subseteq A \quad \text{and} \quad -1 \notin A\}\\
&=& \bigcup _{l=0} ^{k+1} \Big \{A \in S(n;k) : \{l, \ldots k\} \subseteq A \quad \text{and} \quad l-1 \notin A \Big \}.
\end{eqnarray*}
From this partition, it follows that
\begin{equation}\label{pHelpRec1}
\phelp (n;k) = \min_{A \in S(n;k)} \Big\{ per(A) \Big \} = \min_{l \in \{0, 1, \ldots k+1\}} \Bigg \{ \min_{A \in S(n;k)} \Big\{ per(A) : \{l, \ldots k\} \subseteq A \quad \text{and} \quad l-1 \notin A \Big \} \Bigg \}.
\end{equation}
\paragraph*{}Now let $0 \leq l \leq k+1$ be fixed.  Then we have
\begin{eqnarray*}
& & \min_{A \in S(n;k)} \Big\{ per(A) : \{l, \ldots k\} \subseteq A \quad \text{and} \quad l-1 \notin A \Big \}\\
& & \qquad = \min_{A \subseteq \{0,1, \ldots, k\}} \Big\{ per(A) : vol(A)=n, \quad \{l, \ldots k\} \subseteq A \quad \text{and} \quad l-1 \notin A \Big \}\\
& & \qquad = \min_{B \subseteq \{0,1, \ldots, l-2\}} \Big\{ per(B \cup \{l, l+1, \ldots, k\}) : vol(B \cup \{l, l+1, \ldots, k\})=n \Big \}\\
& & \qquad = \begin{cases}\displaystyle \min_{B \in \{0,1, \ldots, k-1\}} \Big\{ per(B) : vol(B)=n \Big \} \quad &\text{if $l = k+1$},\\ \displaystyle \min_{B \in \{0,1, \ldots, k-2\}} \Big\{ k + per(B) : vol(B)=n-k \Big \} \quad &\text{if $l = k$},\\ \displaystyle \min_{B \in \{0,1, \ldots, l-2\}} \Big\{k + l + per(B) : vol(B)=n - \left [k(k+1)/2 - l(l-1)/2 \right] \Big \} \quad &\text{if $0 \leq l < k$},\end{cases}\\
& & \qquad = \begin{cases}\phelp(n;k-1) \quad &\text{if $l = k+1$},\\ k+ \phelp (n-k; k-2) \quad &\text{if $l = k$},\\ k + l + \phelp \big (n - [k(k+1)-l(l-1)]/2 ; l-2 \big ) \quad &\text{if $0 \leq l < k$}.\end{cases}
\end{eqnarray*}
Therefore, by substituting into \eqref{pHelpRec1}, we are able to obtain the recurrence
\begin{eqnarray}
\phelp (n;k) &=& \min \Bigg \{ \phelp(n;k-1), k + \phelp(n;k-2),\nonumber\\
& & \qquad \qquad k+\min_{l \in \{0, \ldots, k-1\}} \Big\{l + \phelp \big (n - [k(k+1)-l(l-1)]/2 ; l-2 \big ) \Big \} \Bigg \}\label{pHelpRec},
\end{eqnarray}
which is valid for all $n \geq 1$ and for all $k \geq 1$.  Moreover, as boundary conditions, which are clear from its definition, we have that $\phelp (n;k)$ satisfies
\[
p(n;k) = \begin{cases}0 \qquad &\text{if $n=0$,}\\ \infty \qquad &\text{if $n < 0$ or $k \leq 0 < n$.}\end{cases}
\]
\paragraph*{}Now because $P(n) = \phelp (n;n)$, this recurrence for $\phelp (n;k)$ gives us the following compact recursive representation for $P(n)$ valid for all $n \geq 1$:
\begin{equation}
P(n) = \min \Big \{p(n; n-1), n \Big \},
\end{equation}
with $P(0) = 0$.

\subsection*{Introduction of $\pinv (n)$ and Derivation of First Recurrences}
Because of its intimate connections with the function $P(n)$ that will be explored in subsequent sections, we now introduce the function $\pinv (n)$, which is defined as
\[
\pinv (n) = \min_{A \subseteq \{0, 1, \ldots\}} \Big \{ per(A^c) : vol(A)=n \Big \}.
\]
The difference between this function and the function $P(n)$ is subtle, and based on how similarly the two functions are defined, one would expect their behavior to be very close.  As we will discuss in subsequent sections, this is indeed the case, and the connections between $P(n)$ and $\pinv(n)$ are actually of fundamental importance.  However, it is important for the reader to keep in mind the differences between these two functions throughout the remainder of the paper.
\paragraph*{}As with the function $P(n)$, we define the auxilary function $\pinvhelp (n;k)$ as
\[
\pinvhelp (n;k) = \min_{A \subseteq \{0, 1, \ldots, k\}} \Big \{ per(A^c): vol(A)=n \Big \},
\]
and just as before, for all $n\geq 0$, we have that
\begin{equation}\label{QRecHelp1}
\pinv (n) = \pinvhelp (n;n).
\end{equation}
\paragraph*{}Now because of the differences between the functions $P(n)$ and $\pinv (n)$, at this point, we must define a special auxilary function, $\pinvequal (n;k)$, in order to determine $\pinv (n)$.  This function is given by
\[
\pinvequal (n;k) = \min_{A \subseteq \{0, 1, \ldots, k\}} \Big \{ per(A^c): vol(A)=n \quad \text{and} \quad k \in A \Big \}.
\]
Note the similarities between $\pinvequal (n;k)$ and $\pinvhelp (n;k)$.  In fact, it is easy to see that for all $n \geq 1$ and $k \geq 0$, we have
\begin{equation}\label{qRec}
\pinvhelp (n; k) = \min_{l \in \{1, 2, \ldots , k\}} \Big \{ \pinvequal (n;l) \Big \}.
\end{equation}
Using this equation and \eqref{QRecHelp1}, we obtain that for all $n \geq 1$
\begin{equation}\label{QRecHelp2}
\pinv (n) = \min_{l \in \{1, 2, \ldots , n\}} \Big \{ \pinvequal (n;l) \Big \},
\end{equation}
with $\pinv (0) = 0$.
\paragraph*{}Just as was the case for $P(n)$, in order to obtain a useful recurrence relation for $Q(n)$, it now only remains to find a recurrence for $\pinvequal (n;k)$.  And as before, we accomplish this by a simple partition yielding
\[
\pinvequal (n;k) = k+1 + \min \Big \{\pinvequal (n-k; k-1) - k, k-1 + \pinvhelp (n-k; k-3), \pinvequal (n-k; k-2)  \Big \},
\]
which we obtain by partitioning the subsets of interest into the three groups (I) sets containing $k-1$, (II) sets containing $k-2$ but not $k-1$, and (III) sets containing neither $k-2$ nor $k-1$.

\paragraph*{}At this point, we need to note that some care must be given to the interpretation of the above equation, which depends on how we define $\pinvequal (0;0)$.  However, if we note and state as a boundary condition that $\pinvequal(n,n) = 2n$ for all $n \geq 1$, then these concerns are effectively removed.
\paragraph*{}We then have a recurrence relation for $\pinvequal$.  As boundary conditions for $\pinvequal (n; k)$, we have
\[
\pinvequal (n;k) = \begin{cases}0 \qquad &\text{if $n=k=0$,}\\ 2n \qquad &\text{if $n=k \geq 1$,}\\ \infty \qquad &\text{if $n < 0$ or if $k \in \{0, 1\}$ and $n > k$,}\\ \infty \qquad &\text{if $0 \leq k > n \geq 0$.}\end{cases}
\]
Then for all $n \geq 2$, and $2 \leq k < n$, we have
\[
\pinvequal (n;k) = k+1 + \min \Big \{k-1 + \pinvhelp (n-k; t-3), \pinvequal (n-k; k-2), \pinvequal (n-k; k-1) - k  \Big \}.
\]
\paragraph*{}Thus, by using \eqref{QRecHelp2} we have a recurrence for $Q(n)$ as well.

\section{More Direct Recurrence Relations}\label{section second recurrences}
The following relations were found by making use of different partitions of the sets of interest.  The discovery of the relation for $P(n)$ was what first motivated the author to study the functions $Q(n)$, $\pinvhelp(n;k)$ and $\pinvequal(n;k)$.
\subsection*{Recurrence for $P(n)$ involving $\pinvhelp (n;k)$ and $\pinvequal(n;k)$}
We also have another recurrence that can be used to calculate $P(n)$ ``more directly".  It is found by partitioning all sets of volume $n$ first according to their maximum element, $m$, and then according to the largest integer smaller than $m$ not contained in this set.
\paragraph*{}Let $A$ be a set of volume $n$ and let $m$ be its maximum element.  Then let $l$ be the largest element of $\{-1, 0, \ldots, m\}$ not contained in $A$.  Then $A$ may be written uniquely as $A = \{0, 1, 2, \ldots, m\} \setminus B$ for some set $B \subseteq \{0, 1, \ldots , l\}$, where the volume of $B$ is equal to $(1 + 2 + \cdots + m)-n$ and $l \in B$.  If $l=m-1$, then $per(A) = per(B^c)$.  Else, we have $per(A) = m + per(B^c)$.
\paragraph*{}From this observation, we obtain that for all $n \geq 2$
\begin{equation}\label{PGoodRec}
P(n) = \min_{m \geq 1} \Big \{m + \pinvhelp ([1 + 2 + \cdots + m] - n; m-2), \pinvequal ([1 + 2 + \cdots + m] - n; m-1)  \Big \},
\end{equation}
where $\pinvhelp (n;k)$ and $\pinvequal (n; k)$ are defined as earlier.

\subsection*{Recurrence for $Q(n)$ involving $\phelp (n;k)$}
As before, we also have a recurrence that can be used to calculate $\pinv (n)$ ``more directly".  It is found by partitioning all sets of volume $n$ first according to their maximum element, which we will denote by $m$.
\paragraph*{}Let $A$ be a set of volume $n$ and maximum element $m$.  Then the set $A$ may be written uniquely in the form $A = \{0, 1, 2, \ldots, m\} \setminus B$ for some set $B \subseteq \{0, 1, \ldots , m-1\}$, where the volume of $B$ is equal to $(1 + 2 + \cdots + m) - n$.  Now we know that for all such sets $A$ and $B$, we have $per(A^c) = per(B) + (m+1)$.
\paragraph*{}This observation leads to the simple and beautiful recurrence that for all $n\geq 2$,
\begin{equation}\label{QGoodRec}
\pinv (n) = 1 + \min_{m \geq 1} \Big \{m + \phelp ([1 + 2 + \cdots + m] - n; m-1) \Big \},
\end{equation}
where $\phelp (n; k)$ is as defined earlier.

\section{Analysis of Recurrences}\label{section analysis of recurrences}
Although equations \eqref{PGoodRec} and \eqref{QGoodRec} may appear at first to be somewhat intractable, they actually are crucial in understanding the behavior of $P(n)$ (and of $Q(n)$ as well).  Attempting to squeeze as much as possible from these two equations, the author was able to obtain many surprising and interesting results, culminating in the proof of our main theorem.  However, since these minor results follow from our main theorem (and the proofs of the minor results are not particularly enlightening), we will solely devote our analysis to proving the main theorem and state these minor results in another section.
\subsection*{Relevant Lemmas and Notions}
The recurrences of equations \eqref{PGoodRec} and \eqref{QGoodRec} came about by attempting to find $P(n)$ for special values of $n$.\footnote{At this point, the author was not interested in the function $Q(n)$ whatsoever.  However, ultimately these musings about $P(n)$ led to the understanding of both functions, as will be discussed.}  For example, if it happens that $n$ may be written in the form
\[
n = 0 + 1 + 2 + \cdots + t = \dfrac{t(t+1)}{2} = Tr(t)
\]
for some integer $t \geq 0$, then it is not difficult to see that $P(n) = t$.  Thus, if $n$ is equal to $\dfrac{t(t+1)}{2} = Tr(t)$ (the $t^{\text{th}}$ triangular number), then we know exactly what $P(n)$ is.
\paragraph*{}This then led to investigating the case $n = [0 + 1 + 2 + \cdots + t] - 1 = \dfrac{t(t+1)}{2} - 1 = Tr(t) - 1$, where $t \geq 2$.  Then by considering the set $\{2, 3, 4, \ldots, t\}$, we see that $P(n) \leq t + 2$.  Moreover, with some determination in manipulating \eqref{PGoodRec}, it can be shown that for all $t \geq 2$, if $n = Tr(t) - 1$ then we have
\[
P(n) = \begin{cases}t, \qquad &\text{if $t=2$,}\\t+2, \qquad &\text{if $t \geq 3$.}\end{cases}
\]
Thus, if $n = Tr(t)$ or $n = Tr(t) - 1$, the behavior of $P(n)$ is known completely, which led to the introduction of the functions $f(n)$ and $g(n)$, as described in the following lemma:
\begin{lemma}
Let $n$ be a positive integer.  Then there exist unique positive integers $f(n)$ and $g(n)$ satisfying
\[
n = [0 + 1 + \cdots + f(n)] - g(n),
\]
where $0 \leq g(n) < f(n)$.  Moreover, $f(n)$ and $g(n)$ are given by\footnote{We will often use these explicit functional representations for $f(n)$ and $g(n)$ so that $f(0) = g(0) = 0$ is well-defined.}
\[
f(n) = \left \lceil \dfrac{-1 + \sqrt{1+8n}}{2} \right \rceil, \qquad \text{and} \qquad g(n) = \dfrac{f(n) [ f(n) + 1]}{2} - n.
\]
\end{lemma}
\begin{proof}
To demonstrate existence, let $n \geq 1$ be arbitrary, and let $F \geq 1$ be the smallest integer satisfying $n \leq 0 + 1 + \cdots + F$.  Then $n$ may be written as $n = [0 + 1 + \cdots + F] - G$, where $0 \leq G$.  Moreover, we know that $G < F$ because otherwise that would contradict the minimality of $F$.
\paragraph*{}We will now show that any positive integers satisfying the claim are necessarily equal to the asserted expresions involving $n$, which will simultaneously establish uniqueness and validate the desired representations, completing the proof.
\paragraph*{}Let $x$ and $y$ be positive integers satisfying $n = [0 + 1 + \cdots + x] - y$ with $0 \leq y < x$.  Then it follows that
\[
0 + 1 + \cdots + (x-1) = \dfrac{x (x-1)}{2} < n = (0+1+\cdots + x) -y \leq 0 + 1 + \cdots	+ x = \dfrac{x(x+1)}{2}.
\]
But after simple manipulation, we see that $x$ and $y$ must necessarily equal the desired expressions.
\end{proof}
\paragraph*{}Having defined these functions, we may now restate previous lemmas involving $P(n)$ and $Q(n)$ in these terms.  The most important result we will use combines \textbf{Propositions \ref{lowerBoundForP}} and \textbf{\ref{lowerBoundForQ}} as follows:
\begin{corollary} \label{crudeLowerBounds}
Restating earlier results in new notation, for all $n \geq 1$, we have that
\[
P(n) \geq f(n), \qquad \text{and} \qquad Q(n) \geq f(n) + 1.
\]
\end{corollary}

\paragraph*{}Finally, before moving on, we present two more results on the functions $f(x)$ and $g(x)$ which will be necessary in the coming sections.

\begin{proposition}\label{fReps}
Let $f(n) = \left \lceil \dfrac{-1+\sqrt{1+8n}}{2} \right \rceil$ as before.  Then for all integers $n \geq 0$, we have
\[
f(n) = \left \lceil \dfrac{-1+\sqrt{1+8n}}{2} \right \rceil = \left \lceil \sqrt{2n} -1/2 \right \rceil = \left[ \sqrt{2n} \right],
\]
where for all $x \in \R$, $[x]$ denotes the nearest integer to $x$.  We will prove that $\sqrt{2n}$ is never a half-integer, which justifies this definition.
\end{proposition}
\begin{proof}
We will start by showing that the first part of the stated equation holds.  We will then show that $f(n) \neq \sqrt{2n}-1/2$ for any integers $n \geq 0$, and that will imply
\[
\sqrt{2n} - 1/2 < f(n) < \sqrt{2n} + 1/2,
\]
which will complete the proof.
\paragraph*{}By way of contradiction, suppose that the first two representations are not equal.  Then since
\[
\sqrt{2n} -1/2 < \dfrac{-1 + \sqrt{1+8n}}{2},
\]
this would imply that there exist integers $p \in \Z$ and $n \in \{0, 1, \ldots \}$ such that
\[
\sqrt{2n} -1/2 \leq p < \dfrac{\sqrt{1+8n} - 1}{2}.
\]
Now multiplying both sides by 2, adding 1, then squaring gives us
\[
8n \leq (2p +1)^2 < 8n+1.
\]
But since $n$ and $p$ are integers, this forces $8n = (2p +1)^2$, which taken modulo 2 yields a contradiction.
\paragraph*{}Furthermore, we know that for all nonnegative integers, $a$, $\sqrt{a}$ is either an integer or irrational.  Therefore, this implies that $\sqrt{2n} - 1/2$ is never an integer for any integers $n\geq0$.  But $f(n) \in \Z$ for all $n$.  Thus $f(n) \neq \sqrt{2n} -1/2$ for any $n$, which completes the proof.
\end{proof}
\begin{proposition}\label{gBound}
For all integers $n \geq 0$, let $g(n) = f(n) [f(n) +1] /2 - n$, where $f(n)$ is defined as before.  Then for all integers $L \geq 0$ and $n \geq 0$, we have
\[
g^{L} (n) \leq 2 \cdot (n/2) ^{1/2^{L}}.
\]
\end{proposition}
\begin{proof}
The proof is by induction on $L$.  If $L=0$, then the claim is trivially true, which establishes the base case.  Now suppose the claim holds for $L = m$.  Then for all $n \geq 0$, we have
\[
g(n) \leq f(n) - 1 < \sqrt{2n} - 1/2 < \sqrt{2n}.
\]
Therefore, we have
\[
g^{m+1}(n) = g( g^{m} (n) ) < \sqrt{2 \cdot g^{m} (n)}.
\]
Since the square root function is increasing, we have by the induction hypothesis
\[
g^{m+1}(n) < \sqrt{2 \cdot g^{m} (n)} < \sqrt{2 \cdot 2 \cdot (n/2) ^{1/2^{m}}} = 2 \cdot (n/2)^{1/2^{m+1}}.
\]
Thus, the claim holds for $m+1$ as well, which completes the proof.
\end{proof}

\subsection*{Upper Bounds and Asymptotics for $P(n)$ and $\pinv (n)$}
The first real gems that we will extract from equations \eqref{PGoodRec} and \eqref{QGoodRec} are simple upper bounds on $P(n)$ and $Q(n)$.  Then using these points and the previous few results, we obtain good absolute bounds on both $P(n)$ and $Q(n)$ in terms of $n$, which reveal the asymptotic behavior of $P(n)$ and $Q(n)$, and which are tighter than those provided in \textbf{Theorem \ref{theirResult}} in \cite{Miller}.
\begin{theorem}\label{goodUpperBounds}
Let $f(n)$ and $g(n)$ be defined as before.  Then for all $n \geq 0$, we have the bounds
\begin{eqnarray*}
P(n) &\leq& f(n) + Q(g(n)), \qquad \text{and}\\
Q(n) &\leq& 1 + f(n) + P(g(n)).
\end{eqnarray*}
\end{theorem}
\begin{proof}
For $n=0$ and $n=1$, the two inequalities hold.  We know from \eqref{PGoodRec} that for all $n \geq 2$,
\begin{eqnarray*}
P(n) &=& \min_{m \geq 1} \Big \{m + \pinvhelp ([1 + 2 + \cdots + m] - n; m-2), \pinvequal ([1 + 2 + \cdots + m] - n; m-1)  \Big \}\\
&\leq & \min \Big \{f(n) + \pinvhelp ([1 + 2 + \cdots + f(n)] - n; f(n)-2), \pinvequal ([1 + 2 + \cdots + f(n)] - n; f(n)-1)  \Big \}\\
&=& \min \Big \{f(n) + \pinvhelp (g(n); f(n)-2), \pinvequal (g(n); f(n)-1)  \Big \}\\
&\leq & f(n) + \min \Big \{\pinvhelp (g(n); f(n)-2), \pinvequal (g(n); f(n)-1) \Big \} = f(n) + \pinvhelp(g(n); f(n)-1).
\end{eqnarray*}
But since $g(n) \leq f(n) -1$, we have that $\pinvhelp( g(n); f(n)-1) = \pinv (g(n))$, which implies that for all $n \geq 2$
\[
P(n) \leq f(n) + \pinv ( g(n)).
\]
Similarly, for all $n \geq 2$, we know from \eqref{QGoodRec} that
\begin{eqnarray*}
\pinv (n) &=& 1 + \min_{m \geq 1} \Big \{m + \phelp ([1 + 2 + \cdots + m] - n; m-1) \Big \}\\
&\leq& 1 + f(n) + \phelp ([1 + 2 + \cdots + f(n)] - n; f(n)-1) = 1 + f(n) + \phelp(g(n); f(n)-1)\\
&=& 1 + f(n) + P(g(n)).
\end{eqnarray*}
Therefore, both inequalities hold for all $n \geq 0$, as desired.
\end{proof}
\paragraph*{}These upper bounds immediately give rise to the following corollaries.
\begin{corollary}
For all nonnegative integers $n$ and $L$, we have that
\begin{eqnarray*}
P(n) &\leq& L + P( g^{2L} (n)) + \sum_{i=0} ^{2L-1} f(g^{i}(n)), \qquad \text{and}\\
\pinv (n) &\leq & L + \pinv( g^{2L} (n)) + \sum_{i=0} ^{2L-1} f(g^{i}(n)),
\end{eqnarray*}
where $g^i (n)$ is the $i$-fold composition of $g$ evaluated at $n$, and by convention we take $g^0 (n) = n$.
\end{corollary}
\begin{proof}
These inequalities are obtained simply by repeatedly appealing to the results of \textbf{Theorem \ref{goodUpperBounds}}.
\end{proof}
\begin{theorem}\label{myBounds}
Let $P(n)$ and $Q(n)$ be as given.  Then we have $P(n) \sim Q(n) \sim \sqrt{2} n^{1/2}$.  Moreover, for all $n > 2$, we have
\begin{eqnarray*}
\sqrt{2} n^{1/2} - 1/2 < &P(n)& \leq \sqrt{2}n^{1/2} + (2^{3/4} \cdot n^{1/4} + 1)[\log_2 ( \log_2 (n/2)) - 1]  + 7, \qquad \text{and}\\
\sqrt{2} n^{1/2} + 1/2 < &Q(n)& \leq \sqrt{2}n^{1/2} + (2^{3/4} \cdot n^{1/4} + 1)[\log_2 ( \log_2 (n/2)) - 1]  + 7.
\end{eqnarray*}
\end{theorem}
\begin{proof}
The lower bounds in the asserted inequalities have already been proven.  To prove the upper bounds, we merely combine the results in the last corollary with several of the previously obtained bounds on $f(n)$ and $g(n)$.
\paragraph*{}More specifically, assuming $n > 2$, we know from \textbf{Lemma \ref{gBound}} that if $L \geq (\log_2 ( \log_2 (n/2)) - 1)/2$, then
\[
g^{2L}(n) \leq 2 \cdot (n/2) ^{1/2^{(\log_2 ( \log_2 (n/2)) - 1)}} = \cdots = 8.
\]
By considering values of $P(n)$ and $Q(n)$ for $n \leq 8$, we see that $g^{2L}(n) \leq 8$ implies $P(g^{2L}(n)) \leq 7$ and $Q(g^{2L}(n)) \leq 7$.  Therefore, we have
\begin{eqnarray*}
P(n) & \leq & L + P(g^{2L}(n)) + \sum_{i=0} ^{2L-1} f(g^{i}(n))\\
& \leq & L + P(g^{2L}(n)) + \sum_{i=0} ^{2L-1} \sqrt{2 g^{i}(n)} + 1/2\\
& \leq & 2L + P(g^{2L}(n)) + \sum_{i=0} ^{2L-1} \sqrt{4 \cdot (n/2) ^{1/2^{i}}}\\
& \leq & 2L + P(g^{2L}(n)) + \sqrt{2n} + 2 \sum_{i=1} ^{2L-1} \sqrt{(n/2) ^{1/2^{i}}}\\
& \leq & 2L + P(g^{2L}(n)) + \sqrt{2n} + 4L(n/2) ^{1/4}.
\end{eqnarray*}
Then taking $L=(\log_2 ( \log_2 (n/2)) - 1)/2$, we have that $P(g^{2L}(n)) \leq 7$, which gives us the desired bound on $P(n)$, and the bound on $Q(n)$ is obtained in the same way.
\end{proof}
Note that these bounds on $P(n)$ are slightly better than those of \cite{Miller} stated in \textbf{Theorem \ref{theirResult}}.  Also note that the upper bound on the summation is very crude.  These bounds are sufficiently good for our purposes, so we will leave them as they are.

\section{Obtaining \textit{Good} Recurrences for $P(n)$ and $Q(n)$} \label{section good recurrences}
Although the bounds in \textbf{Theorem \ref{myBounds}} are quite good, they reveal nothing about the actual fluctuations of $P(n)$ and $Q(n)$.  And although we have already obtained multiple recurrence relations for finding exact values, these relations all involve auxilary helper functions, multiple variables, and unweildy minimum functions.  In this section, we will remedy this by providing simple and satisfying recurrences and even explicit formulae.
\paragraph*{}First, we will again make use of equations \eqref{PGoodRec} and \eqref{QGoodRec} to provide new lower bounds on $P(n)$ and $Q(n)$.  By appealing to the analytic bounds of \textbf{Theorem \ref{myBounds}}, we will then show that for all sufficiently large $n$, these lower bounds (surprisingly) simplify and coincide with the upper bounds we provided in \textbf{Theorem \ref{goodUpperBounds}} yielding an elegant representation for the functions.  We conclude the section by noting some corollaries of this result and describing some quasi-explicit formulae.

\subsection*{Lower Bounds on $P(n)$ and $Q(n)$}
Before deriving our lower bounds, we briefly note the following lemmas.
\begin{lemma}\label{infinityBound}
Let $n$ and $k$ be positive integers with $k < f(n)$.  Then $\phelp(n;k),$ $\pinvhelp(n;k)$, and $\pinvequal(n;k)$ are all infinite.
\end{lemma}
\begin{proof}
This follows immediately from the fact that if $k < f(n)$, then there are no subsets of $\{0, 1, \ldots, k\}$ having volume $n$.
\end{proof}

\begin{lemma}\label{helperBound}
Let $n$ and $m$ be positive integers with $m > f(n)$.  Then we have
\begin{eqnarray*}
m + \phelp ([1 + 2 + \cdots + m] - n; m-1) &\geq& f(n) + \sqrt{2 (g(n) + f(n)+1)} + 1/2 \qquad \qquad \text{and}\\
m + \pinvhelp ([1 + 2+\cdots +m] - n; m-2) &\geq& f(n) + \sqrt{2 (g(n) + f(n)+1)} + 3/2.
\end{eqnarray*}
\end{lemma}
\begin{proof}
These inequalities both follow directly from the simple lower bounds used in \textbf{Theorem \ref{myBounds}}.  Consider the following chain of inequalities
\begin{eqnarray*}
\phelp ([1 + 2+\cdots +m] - n; m-1) &\geq& P ([1 + 2+\cdots +m] - n) \geq \sqrt{2 ([1 + 2+\cdots +m] - n)} - 1/2\\
&\geq& \sqrt{2 (g(n) + [f(n)+1] + [f(n)+2] + \cdots + m)} - 1/2\\
&\geq& \sqrt{2 (g(n) + f(n)+1)} - 1/2.
\end{eqnarray*}
Then adding $m \geq f(n) +1$ to both sides completes the proof.  The second inequality is proven in the same way.
\end{proof}

\begin{lemma}\label{piBound}
Let $n$ and $m$ be positive integers with $m \geq f(n)$.  Then we have
\[
\pinvequal ([1 + 2 + \cdots + m] - n; m-1) \geq 2 f(n)-2.
\]
\end{lemma}
\begin{proof}
First, we may assume $f(n) \geq 2$, because otherwise the claim is trivially true.  Now let $A$ be a subset of $\{0, 1, 2, \ldots, m-1\}$ such that $A$ has volume $[1+2+\cdots + m] -n$ and $m-1 \in A$.  By way of contradiction, suppose that $per(A^c) < 2f(n) -2$.
\paragraph*{}If $m \geq 2f(n) - 2$, then since $m-1 \in \partial A$, this would imply that $per(A^c) \geq m \geq 2f(n) -2$.  Therefore, we may assume that $m \leq 2f(n) - 3$.
\paragraph*{}Now since $m \geq f(n)$, the volume of $A$ may be written as
\begin{eqnarray*}
vol(A) &=& [1 + 2 + \cdots + m] - n = g(n) + [(f(n) + 1) + (f(n) + 2) + \cdots + m]\\
 &<& f(n) + [f(n) + 1] + \cdots + m,
\end{eqnarray*}
and because $m \leq 2f(n) -3 = [f(n) - 2] + [f(n) - 1]$, we also have
\[
vol(A) < [f(n)] + [f(n) + 1] + \cdots + [m - 1] + [f(n)-2] + [f(n) -1] = \sum_{i=f(n)-2}^{m-1} i.
\]
From this, we konw that there is at least one element of $\{f(n)-2, f(n)-1, \ldots , m-2\}$ that not contained in $A$, because otherwise the volume of $A$ would be too large.
\paragraph*{}Now let $l \in A^c$ be the largest integer satisfying $f(n)-2 \leq l \leq m-2$.  Then since $m-1 \in A$, we know that $l \in \partial A^c$, which implies
\[
per(A^c) \geq l + m \geq f(n) -2 + m \geq f(n) - 2 + f(n) = 2f(n) -2.
\]
But this contradicts the assumption that $per(A^c) < 2f(n) -2$, thus completing the proof.
\end{proof}
\paragraph*{}With these lemmas, we are now able to prove the following lower bounds.
\begin{theorem}\label{goodLowerBounds}
Let $P(n)$ and $Q(n)$ be as given.  Then for all $n \geq 2$, we have
\begin{eqnarray*}
P(n) &\geq & f(n) + \min \Big \{\pinv (g(n)), \sqrt{2 (g(n) + f(n)+1)} + 3/2, f(n)-2 \Big \} \qquad \qquad \text{and}\\
Q(n) &\geq & 1 + f(n) + \min \Big \{P (g(n)), \sqrt{2 (g(n) + f(n)+1)} + 1/2 \Big \}.
\end{eqnarray*}
\end{theorem}
\begin{proof}
Starting with \eqref{PGoodRec} and applying \textbf{Lemmas \ref{infinityBound}, \ref{helperBound},} and \textbf{\ref{piBound}}, we obtain
\begin{eqnarray*}
P(n) &=& \min_{m \geq 1} \Big \{m + \pinvhelp ([1 + 2+\cdots +m] - n; m-2), \pinvequal ([1 + 2+\cdots +m] - n; m-1)  \Big \}\\
&=& \min_{m \geq f(n)} \Big \{m + \pinvhelp ([1 + 2+\cdots +m] - n; m-2), \pinvequal ([1 + 2+\cdots +m] - n; m-1)  \Big \}\\
&=& \min_{m > f(n)} \Big \{f(n) + \pinvhelp(g(n); f(n)-2), m + \pinvhelp ([1 + 2+\cdots +m] - n; m-2),\\
& & \qquad \qquad \pinvequal(g(n); f(n)-1), \pinvequal ([1 + 2+\cdots +m] - n; m-1)  \Big \}\\
&\geq& f(n) + \min \Big \{\pinv (g(n)), \sqrt{2 (g(n) + f(n)+1)} + 3/2, f(n)-2 \Big \}.
\end{eqnarray*}
The second inequality is proven analogously by starting with \eqref{QGoodRec}.
\end{proof}
\ignore{
Similarly, we also may obtain the following lower bound on $Q(n)$.
\begin{proposition}\label{goodLowerBoundForQ}
Let $n \geq 2$.  Then $Q(n)$ satisfies
\[
Q(n) \geq 1 + f(n) + \min \Big \{P (g(n)), \sqrt{2 (g(n) + f(n)+1)} + 1/2 \Big \}.
\]
\end{proposition}
\begin{proof}Just as before, starting from \eqref{QGoodRec}, for all $n \geq 2$ we have
\begin{eqnarray*}
\pinv (n) &=& 1 + \min_{m \geq 1} \Big \{m + \phelp ([1 + 2 + \cdots + m] - n; m-1) \Big \}\\
&=& 1 + \min_{m \geq f(n)} \Big \{m + \phelp ([1 + 2 + \cdots + m] - n; m-1) \Big \}\\
&\geq & 1 + \min_{m > f(n)} \Big \{f(n) + \phelp (g(n); f(n)-1), m + \phelp ([1 + 2 + \cdots + m] - n; m-1) \Big \}\\
&\geq & 1 + f(n) + \min_{m > f(n)} \Big \{P(g(n)), \sqrt{2 (g(n) + f(n)+1)} + 1/2 \Big \},
\end{eqnarray*}
as desired.
\end{proof}
}

\subsection*{Squeezing an Equation from Inequalities (Eventually)}
At this point, we have simple upper bounds on $P(n)$ and $Q(n)$ provided by \textbf{Theorem \ref{goodUpperBounds}}, and nearly simple lower bounds from \textbf{Theorem \ref{goodLowerBounds}}, which are complicated by the ``min" operators.  Suppse we could show that \textit{eventually} $P(g(n))$ and $Q(g(n))$ happen to be the smallest terms in each minimum.  Then our lower bounds would simplify drastically and our lower and upper bounds would squeeze together, yielding a simple pair of recursive equations that would hold for all sufficiently large $n$.
\paragraph*{}As luck would have it, we can in fact prove that eventually $P(g(n))$ and $Q(g(n))$ are the smallest terms in each minimum, as shown in the following proposition.
\begin{proposition}\label{eventuallyHappens}
Let $P(n)$ and $Q(n)$ be as given.  Then there exists $N \in \mathbb{Z}$ such that for all $n \geq N$
\begin{eqnarray*}
P(g(n)) &=& \min \Big \{P (g(n)), \sqrt{2 (g(n) + f(n)+1)} + 1/2 \Big \} \qquad \qquad \text{and}\\
Q(g(n)) &=& \min \Big \{\pinv (g(n)), \sqrt{2 (g(n) + f(n)+1)} + 3/2, f(n)-2 \Big \}.
\end{eqnarray*}
\end{proposition}
\begin{proof}
We will prove that there exists some value $N_P$ after which the first equation holds.  The proof that there also exists a value $N_Q$ after which the second equation holds is virtually identical.  Then taking $N = \max \{N_P , N_Q\}$ will complete the proof.
\paragraph*{}We need to show that eventually $P(g(n)) \leq \sqrt{2 (g(n) + f(n)+1)} + 1/2$.  From \textbf{Theorem \ref{myBounds}}, we know that
\[
P(r) \leq \sqrt{2r} + \littleO(\sqrt{r}).
\]
Therefore, there exists a constant $G$ such that for all $r \geq G$, we have
\[
P(r) \leq \sqrt{2r} + \littleO(\sqrt{r}) \leq \sqrt{4r}.
\]
From this, it follows that for all $n$, if $g(n) \geq G$, then we have
\[
P(g(n)) \leq \sqrt{4g(n)} \leq \sqrt{2(g(n) + f(n) + 1)} + 1/2.
\]
\paragraph*{}Let $M$ be the maximum value taken by $P(k)$ for $0 \leq k \leq G$, and let $n \geq M^2 (M^2+1)/2$ be arbitrary.  Now if $g(n) \geq G$, then we know the claim holds.  Therefore, we can assume $g(n) < G$.  But if this is that case, then we know $P(g(n)) \leq M$, which implies
\[
P(g(n)) \leq M \leq \sqrt{f(n)} \leq \sqrt{2(g(n) + f(n)+1)}+1/2.
\]
\paragraph*{}Therefore, for all $n \geq M^2 (M^2+1)/2$, the first equation holds.  Thus, by our previous remarks, this completes the proof.
\end{proof}
With this proposition, we are able to prove our main result.
\begin{theorem}\label{bestResult}
Let $P(n)$ and $Q(n)$ be as given.  Then there exists some integer $N$ such that for all $n\geq N$, 
\begin{eqnarray*}
P(n) &=& f(n) + Q(g(n)) \qquad \qquad \text{and}\\
Q(n) &=& 1 + f(n) + P(g(n)),
\end{eqnarray*}
where $f(n)$ and $g(n)$ are the previously defined functions.
\end{theorem}
\begin{proof}
This follows readily by using the previous proposition to simply the lower bounds in \textbf{Theorem \ref{goodLowerBounds}} and comparing these to the upper bounds in \textbf{Theorem \ref{goodUpperBounds}}.
\end{proof}

\subsection*{Corollaries and Remarks}
There are many interesting implications of \textbf{Theorem \ref{bestResult}}.  From this result, many things can be discovered about the behavior of $P(n)$ and $Q(n)$, and the intimate connection between these two functions is made evident.
\paragraph*{}Geometrically, we see that there is a certain fractal-like property of the graphs of these functions.  As before, let $Tr(t) = 0 + 1 + \cdots + t$ denote the $t^{\text{th}}$ triangular number.  Then for all sufficiently large $t$, the sequence $\{ P(Tr(t)), P(Tr(t)-1), P(Tr(t)-2), \ldots \}$ is equal to $\{t + Q(0), t + Q(1), t + Q(2), \ldots\}$.
\paragraph*{}In words, for all sufficiently large fixed values of $t$, the graph of $(Tr(t)-n, P(n) - t)$ coincides with the graph of $(n, Q(n))$ for all $0 \leq n \leq t-1$.  Thus, the graph of $P(n)$ eventually consists solely of partial copies of $Q(n)$.  Moreover, we see the graph of $Q(n)$ eventually consists solely of partial copies of $P(n)$.
\paragraph*{}This mutual similarity of the two functions also induces self-similarity as shown in the following results.
\begin{corollary}
For all sufficiently large values of $n$, if $g(n) \neq f(n)-1$, we have
\begin{eqnarray*}
P(n) &=& 1 + P(n-f(n)) \qquad \qquad \text{and}\\
Q(n) &=& 1 + Q(n-f(n)).
\end{eqnarray*}
\end{corollary}
\begin{proof}
This follows from \textbf{Theorem \ref{bestResult}} and the fact that if $g(n) \neq f(n)-1$, then $g(n) = g(n-f(n))$.
\end{proof}

\begin{corollary}
For all values of $n$ such that $g(n)$ is sufficiently large, we have
\begin{eqnarray*}
P(n) &=& 1 + f(n) + f(g(n)) + P(g^2(n)) \qquad \qquad \text{and}\\
Q(n) &=& 1 + f(n) + f(g(n)) + Q(g^2(n)).
\end{eqnarray*}
\end{corollary}
\begin{proof}
This follows immediately by applying \textbf{Theorem \ref{bestResult}} twice.
\end{proof}

\paragraph*{}This last recurrence is readily `solved' yielding the following quasi-explicit equations.
\begin{proposition}\label{almostExplicit}
Let $N$ be as described in \textbf{Theorem \ref{bestResult}}, and let $n \geq 0$ be arbitrary.  Let $\phi(n;N) = \phi(n)$ denote the smallest nonnegative integer satisfying $g^{\phi(n)}(n) \leq N$.  Then we have
\begin{eqnarray*}
P(n) &=& \begin{cases} P(g^{\phi(n)}(n)) + \sum_{i=1}^{\phi(n)} f(g^{i-1}(n)) + \phi(n)/2 \qquad &\text{if $\phi(n)$ is even}\\ Q(g^{\phi(n)}(n)) + \sum_{i=1}^{\phi(n)} f(g^{i-1}(n)) +[\phi(n)-1]/2 \qquad &\text{if $\phi(n)$ is odd,}\end{cases} \qquad \qquad \text{and}\\
Q(n) &=& \begin{cases} Q(g^{\phi(n)}(n)) + \sum_{i=1}^{\phi(n)} f(g^{i-1}(n)) + \phi(n)/2 \qquad &\text{if $\phi(n)$ is even}\\ P(g^{\phi(n)}(n)) + \sum_{i=1}^{\phi(n)} f(g^{i-1}(n)) +[\phi(n)+1]/2 \qquad &\text{if $\phi(n)$ is odd.}\end{cases} 
\end{eqnarray*}
\end{proposition}
\begin{proof}
This follows easily from the previous corollary.  Although the function $\phi(n;N) = \phi(n)$ is much too elusive for most honest mathematicians to call these equations truly ``explicit", they ought not be considered recursive.  This is because even though $P$ and $Q$ are referenced on the right-hand side, their arguments are bounded; therefore, those terms are effectively known.
\end{proof}

\section{Conclusion} \label{section conclusion}
We conclude by discussing how large the value of $n$ must be for our results to hold and by listing some open questions.
\subsection*{``Sufficiently Large" and Computer Algorithms}
In \textbf{Proposition \ref{eventuallyHappens}}, we state results that hold for all sufficiently large values of $n$ without any discussion of how large ``sufficiently large" is.  Although this question can be partially answered by carefully disecting the arguments and bounds used, that approach alone is insufficient.
\paragraph*{}Suppose for instance, that this sort of reasoning reveals that our results hold for all $n \geq 3,000,000$.  Then that is perhaps interesting, but if it turns out that the results actually hold for all $n \geq 97$, then the first statement would seem absurdly high.  Thus, after analytically obtaining the first bound of say $3,000,000$, one would ideally check all values of $P(n)$ and $Q(n)$ for $n \leq 3,000,000$ in order to make the result as strong as possible, which brings up a brief discussion of algorithms.
\paragraph*{}The most na\"ive approach to compute $P(n)$ would be simply to list all sets of volume $n$ and find which has the smallest perimeter.  This would require roughly $\bigO{2^n}$ time and $\bigO{n}$ memory, which is much too slow for large $n$, and a different approach is needed\footnote{Trying to find a more efficient algorithm is actually what initially motivated the first few recurrences in this paper.}.
\paragraph*{}After deriving the recurrence relations for the auxilary functions such as \eqref{pHelpRec} and \eqref{qRec}, the author was able to take advantage of dynamic programming to design algorithms for computing $P(n)$ and $Q(n)$ taking $\bigO{n^2 f(n)} = \bigO{n^{2.5}}$ time and using $\bigO{n^2}$ memory.  By then employing a custom data structure, the author was able to reduce the memory requirement to roughly $\bigO{n}$.

\paragraph*{}Using these algorithms, the author was able to check all values of $P(n)$ and $Q(n)$ for $n \leq 3,500,000$.  In particular, the author checked to see for what values of $n$ the equations in \textbf{Theorem \ref{bestResult}} were valid.  The equations held for all values of $n \leq 3,500,000$ with the exception of a few hundred numbers, the largest of which was $149,894$.

\paragraph*{}Relying both on these computed values and on analytic bounds on $P(n)$ and $Q(n)$, the author very strongly believes that $149,894$ is the largest exception to the equations in \textbf{Theorem \ref{bestResult}}.  However, because this result depends so heavily on computers the author feels obliged to call this merely a ``very plausible conjecture."

\subsection*{Open Questions}
There are several possible areas of future research.  Because the function $P(n)$ was first introduced so recently, this paper serves as a comprehensive overview of all that is known.  The author is more than willing to provide anyone interested with his code and calculated results.
\begin{itemize}
\item[--] Little is known about the behavior of the functions $\phelp(n;k)$, $\pinvhelp(n;k)$, and $\pinvequal(n;k)$.
\item[--] It appears that for any fixed $n \leq 100,000$ the function $\phelp(n;k)$ takes at most two finite values as $k$ varies.  This may be interesting and might be proveable by focusing on \textbf{Proposition \ref{eventuallyHappens}}.
\item[--] The conjecture that $149,894$ is the largest exception to \textbf{Theorem \ref{bestResult}} needs a rigorous proof.
\item[--] In all liklihood, very little or nothing whatsoever is known about $\phi(n;N)$ from \textbf{Proposition \ref{almostExplicit}}.
\item[--] Characterizing sets for which $P(n)$ is obtained may be interesting.  It seems likely that the partitions used and the code developed in this paper would help with that.
\item[--] Providing more direct (i.e., less analytic) proofs for these results would likely be very enlightening.
\end{itemize}
\bibliography{mybib}
\end{document}